\documentclass[11pt]{article}
\usepackage{amsmath,amsthm,amssymb,color}

\def\titlerunning#1{\gdef\titrun{#1}}
\makeatletter
\def\author#1{\gdef\autrun{\def\and{\unskip, }#1}\gdef\@author{#1}}
\def\address#1{{\def\and{\\\hspace*{18pt}}\renewcommand{\thefootnote}{}%
\footnote {#1}}%
\markboth{\autrun}{\titrun}}
\makeatother
\def\email#1{\hspace*{4pt}{\em e-mail}: #1}
\def\MSC#1{{\renewcommand{\thefootnote}{}%
\footnote{\emph{Mathematics Subject Classification (2010):} #1}}}
\def\keywords#1{\par\medskip
\noindent\textbf{Keywords:} #1}

\newtheorem{theorem}{Theorem}[section]
\newtheorem{prop}[theorem]{Proposition}
\newtheorem{cor}[theorem]{Corollary}
\newtheorem{lemma}[theorem]{Lemma}

\theoremstyle{definition}

\newtheorem{cons}[theorem]{Construction}
\newtheorem{remark}[theorem]{Remark}

\numberwithin{equation}{section}

\def\cL{\mathcal L}

\def\cA{\mathcal A}
\def\cC{\mathcal C}

\def\cB{\mathcal B}
\def\cD{\mathcal D}

\def\cF{\mathcal F}
\def\cG{\mathcal G}
\def\cH{\mathcal H}

\def\cM{\mathcal M}
\def\cN{\mathcal N}
\def\cK{\mathcal K}
\def\cO{\mathcal O}
\def\cP{\mathcal P}
\def\cX{\mathcal X}
\def\cY{\mathcal Y}

\def\cV{\mathcal V}
\def\cT{\mathcal T}
\def\cR{\mathcal R}
\def\cS{\mathcal S}
\def\cW{\mathcal W}
\def\cQ{\mathcal Q}
\def\cZ{\mathcal Z}
\def\PG{{\rm PG}}

\def\GF{{\rm GF}}

\def\PGL{{\rm PGL}}

\def\GL{{\rm GL}}

\begin{document}


\baselineskip=16pt

\titlerunning{}

\title{Subspace code constructions}

\author{Antonio Cossidente
\and
Giuseppe Marino
\and
Francesco Pavese}

\date{}

\maketitle

\address{A. Cossidente: Dipartimento di Matematica, Informatica ed Economia, Universit{\`a} degli Studi della Basilicata, Contrada Macchia Romana, 85100, Potenza, Italy; \email{antonio.cossidente@unibas.it}
\and
F. Pavese: Dipartimento di Meccanica, Matematica e Management, Politecnico di Bari, Via Orabona 4, 70125 Bari, Italy; \email{francesco.pavese@poliba.it}
\and
G. Marino: Dipartimento di Matematica e Applicazioni ``Renato Caccioppoli'', Universit{\`a} degli Studi di Napoli ``Federico II'', Complesso Universitario di Monte Sant'Angelo, Cupa Nuova Cintia 21, 80126, Napoli, Italy; \email{giuseppe.marino@unina.it}}

\MSC{Primary 51E20; Secondary 05B25, 94B65}


\begin{abstract}
We improve on the lower bound of the maximum number of planes of $\PG(8,q)$ mutually intersecting in at most one point leading to the following lower bound: $\cA_q(9, 4; 3) \ge q^{12}+2q^8+2q^7+q^6+q^5+q^4+1$ for constant dimension subspace codes.  We also construct two new non--equivalent $(6, (q^3-1)(q^2+q+1), 4; 3)_q$--constant dimension subspace orbit--codes.
\end{abstract}

\section{Introduction}

Let $V$ be an $n$--dimensional vector space over $\GF(q)$, $q$ any prime power. The set $S(V)$ of all subspaces of $V$, or subspaces of the projective space $\PG(V)=\PG(n-1,q)$, forms a metric space with respect to the {\em subspace distance} defined by $d(U,U') = \dim (U+U') - \dim (U\cap U')$. In the context of subspace codes, the main problem is to determine the largest possible size of codes in the space $(S(V), d)$ with a given minimum distance, and to classify the corresponding optimal codes. The interest in these codes is a consequence of the fact that codes in the projective space and codes in the Grassmannian over a finite field referred to as subspace codes and constant--dimension codes, respectively, have been proposed for error control in random linear network coding, see \cite{KK}.

An $(n, M, 2\delta; k)_q$ constant--dimension subspace code (CDC) is a set $\cal C$ of $k$--subspaces of $V$ with $\vert{\cal C}\vert = M$ and minimum subspace distance $d({\cal C}) = \min\{d(U,U') \; \vert \; U,U'\in {\cal C}, U \ne U' \} = 2\delta$. In the terminology of projective geometry, an $(n, M, 2\delta; k)_q$ constant--dimension subspace code, $\delta > 1$, is a set $\cC$ of $(k - 1)$--dimensional projective subspaces of $\PG(n - 1, q)$ such that $|\cC| = M$ and every $(k-\delta)$--dimensional projective subspace of $\PG(n-1, q)$ is contained in at most one member of $\cC$ or, equivalently, any two distinct codewords of $\cC$ intersect in at most a $(k -\delta -1)$--dimensional projective space. The maximum size of an $(n, M, 2\delta; k)_q$ CDC subspace code is denoted by $\cA_q(n, 2\delta; k)$ and several upper bounds are known \cite{XF}, \cite{tables}. The following is known as {\em Johnson bound}:
$$
\cA_q(n, 2\delta; k) \le \frac{(q^n - 1) \cdot \ldots \cdot (q^{n-k+\delta+1} - 1)}{(q^k-1) \cdot \ldots \cdot (q^{\delta+1} - 1)} \cA_q(n - k + \delta, 2\delta, \delta).
$$
As for lower bounds, in \cite{SKK} there is a construction of CDC obtained by using maximum rank distance codes, which yields the bound $\cA_q(n, 2\delta; k) \ge q^{(n-k)(k-\delta+1)}$. Recently subspace codes have been widely investigated by many authors and different approaches of constructing constant dimension codes have been considered.

This paper deals with CDC and it is organized as follows. In Section \ref{1} we describe the geometric background;
Section \ref{2} contains the main result of the paper concerning an improvement on the lower bound of the maximum number of planes of $\PG(8,q)$ mutually intersecting in at most one point. The key idea is to refine the construction method of CDC introduced in \cite{GT} ({\em linkage}) by
using properties of a CDC in $\PG(5,q)$ constructed in \cite{CP}. In this regard, we construct a set of $q^{12} + 2q^8 + 2q^7 + q^6 + q^5 + q^4 +1$ planes of $\PG(8, q)$ mutually intersecting in at most a point providing an improvement on the known lower bound given by $q^{12}+2q^8+2q^7+q^6+1$.
Finally, in Section \ref{3} we are concerned with {\em orbit codes}, i.e.,  constant dimension codes which admit a certain automorphism group and whose members form an orbit under the action of such a group \cite{T, TMR, CRS}. Precisely, we present two constructions of orbit--codes with parameters
$(6,(q^3-1)(q^2+q+1),4;3)_q$. For the first construction we use the set of planes of $\PG(5,q)$ that are tangent to some Veronese surface of the mixed partition of $\PG(5,q)$ constructed in \cite{BBCE}. For the second construction we use a net of Klein quadrics of $\PG(5, q)$ sharing two disjoint planes by suitably selecting Greek planes on each quadric of the net.

We will use the term $n$--space to denote an $n$--dimensional projective subspace of the ambient projective space. We shall consider a point of the ambient projective space as column vectors, with matrices acting on the left. 

\section{Preliminaries}\label{1}

A {\em (non--degenerate) quadric} of a finite projective space is the set of points defined by a non--degenerate quadratic form. A quadric of $\PG(2, q)$ is a set of $q+1$ points no three on a line, called {\em conic}, while a quadric of $\PG(3,q)$ is either {\em hyperbolic} or {\em elliptic}. A hyperbolic quadric ${\cal Q}^+(3,q)$ of $\PG(3,q)$ consists of $(q+1)^2$ points of $\PG(3,q)$ and $2(q+1)$ lines that are the union of two reguli. A {\em regulus} is the set of lines intersecting three skew lines and has size $q+1$. Through a point of ${\cal Q}^+(3,q)$ there pass two lines belonging to different reguli. A plane of $\PG(3,q)$ is either secant to ${\cal Q}^+(3,q)$ and meets ${\cal Q}^+(3,q)$ in a conic or it is tangent to ${\cal Q}^+(3,q)$ and meets ${\cal Q}^+(3,q)$ in two concurrent lines. A hyperbolic quadric $\cQ^+(5,q)$ of $\PG(5,q)$ contains $(q^2+1)(q^2+q+1)$ points. The projective subspaces of maximal dimension contained in $\cQ^+(5,q)$ are planes. In particular there are $2(q+1)(q^2+1)$ planes belonging to $\cQ^+(5,q)$ and they are partitioned in two classes of the same size, called {\em Greek planes} and {\em Latin planes}, respectively. Two distinct planes in the same class meet precisely in one point, while planes lying in different classes are either disjoint or they meet in a line. For more details on quadrics we refer to \cite{h, H, HT}. 

The {\em pencil} of quadrics of $\PG(n, q)$ generated by two quadrics with equations $F = 0$ and $F' = 0$, respectively, is the set of quadrics defined by $\lambda F + \mu F'$, where $\lambda, \mu \in \GF(q)$, $(\lambda, \mu) \ne (0, 0)$. A {\em linear system} of quadrics of $\PG(n, q)$ is a collection $\cB$ of quadrics of $\PG(n, q)$ such that the pencil generated by two quadrics of $\cB$ is contained in $\cB$. 

A cyclic group of $\PGL(3, q)$ permuting points (lines) of $\PG(2,q)$ in a single orbit is called a {\em Singer cyclic group} of $\PGL(3, q)$. A generator of a Singer cyclic group is called a {\em Singer cycle}. See \cite{Huppert}. 

A {\em projective bundle} of $\PG(2,q)$ is a family of $q^2+q+1$ non--degenerate conics of $\PG(2,q)$ mutually intersecting in a point. In other words, the conics in a projective bundle play the role of lines in $\PG(2,q)$, i.e., it is a model of projective plane. Let $\pi$ be a projective plane over $\GF(q)$. Embed $\pi\cong \PG(2,q)$ into $\Pi=\PG(2,q^3)$. Then $\pi={\rm Fix}(\tau)$, where $\tau$ is a collineation of $\Pi$ of order 3. Let us fix a triangle $\Delta$ of vertices $P$, $P^\tau$, $P^{\tau^2}$ in $\Pi\setminus\pi$. Up to date, the known types of projective bundles are as follows \cite{DGG}:

	\begin{itemize}
\item[1.] {\em Circumscribed bundle} consisting of all conics of $\pi$ whose extensions over $\GF(q^3)$ contain the vertices of $\Delta$. This exists for all $q$.

\item[2.] {\em Inscribed bundle} consisting of all conics of $\pi$  whose extensions over $\GF(q^3)$ are tangent to the three sides of $\Delta$. This exists for all odd $q$.

\item[3.] {\em Self--polar bundle} consisting of all conics of $\pi$  whose extensions over $\GF(q^3)$ admit $\Delta$ as a self--polar trangle. This exists for all odd $q$.	
	\end{itemize}
	
Since the triangle $\Delta$ is fixed by a Singer cyclic group of $\PGL(3,q)$, we may conclude that all these projective bundles are invariant under a Singer cyclic group of $\PGL(3,q)$. The first and third type are linear systems of conics, whereas the inscribed bundle is not a linear system. For more details on projective bundles, see \cite{BBEF} and references therein.

\section{On planes of $\PG(8,q)$ mutually intersecting in at most one point}\label{2}

In this section we first investigate the circumscribed bundle of $\PG(2,q)$ and basing on these results  we improve on the lower bound of the maximum number of planes of $\PG(8,q)$ mutually intersecting in at most one point. 

\subsection{The geometric setting}

In \cite[Theorem 2.1]{CP} it has been proved that any projective bundle of $\PG(2,q)$ gives rise to a collection $\cS$ of $q^6+2q^2+2q+1$ planes of $\PG(5,q)$ pairwise intersecting in at most one point. In what follows we briefly summarize how $\cS$ can be obtained. Let $\cB$ be a bundle of a projective plane $\pi=\PG(2,q)$ left invariant by the Singer group $C_1$ and consider $\pi$ as a hyperplane section of $\PG(3,q)$. The stabilizer of $\pi$ in $\PGL(4, q)$, say $\bar{G}$, has structure $q^3:\GL(3,q)$. Let $C_2$ be the subgroup of $\bar{G}$ which fixes pointwise $\pi$. Thus $|C_2| = q^3(q-1)$. By considering $C_1$ as a subgroup of $\bar{G}$, let $G$ be the group generated by $C_1$ and $C_2$. Then $G$ has order $q^3(q^3-1)$, it permutes in a single orbit the points of $\PG(3,q) \setminus \pi$ and $G|_{\pi} = C_1$.   

There are $q^3(q-1)/2$ hyperbolic quadrics of $\PG(3, q)$ meeting $\pi$ in a conic and they are permuted in a single orbit by the group $C_2$. Hence there is a set of $q^3(q^3-1)/2$ hyperbolic quadrics of $\PG(3, q)$, say $\cH$, meeting $\pi$ in a conic of $\cB$. In particular $G$ acts transitively on $\cH$. Consider the Klein correspondence $\rho$ between the lines of $\PG(3,q)$ and the points of the Klein quadric $\cal K$ of $\PG(5, q)$, see \cite{H}. Let $\perp$ be the polarity associated with $\cK$. Under the map $\rho$ the lines of the plane $\pi$ correspond to the points of a Greek plane of $\cal K$, say $\gamma$, the lines of a regulus of $\PG(3,q)$ are mapped to the points of a conic of $\cK$ and the reguli of a hyperbolic quadric of $\PG(3, q)$ correspond to two conics of $\cK$ lying in the planes $\sigma, \sigma^\perp$. The plane defined by the image of a regulus of $\PG(3, q)$ under the map $\rho$ meets $\cK$ in exactly a conic. The image under $\rho$ of the reguli contained in the quadrics of $\cal H$ gives rise to a set $\cS_1$ of $q^6-q^3$ planes of $\PG(5,q)$ meeting $\cal K$ in a conic. The planes in $\cS_1$ are disjoint from $\gamma$ and pairwise meet in at most one point. Note that, since a plane of $\cS_1$ does not contain a line of $\cK$, a plane of $\cK$ and a plane of $\cS_1$ meet in at most one point. Hence $\cS_1$ can be enlarged in two ways.

\begin{cons}\label{c1}
Let $\cS_2$ be that set of $q^6$ planes of $\PG(5, q)$ obtained by adding to $\cS_1$ the set of $q^3$ Latin planes of $\cK$ that are disjoint from $\gamma$. Then every element of $\cS_2$ is disjoint from $\gamma$ and distinct planes in $\cS_2$ meet in at most a point.  
\end{cons}

\begin{cons}\label{c2}
Let $\cS_3$ be that set of $q^6+q^2+q$ planes of $\PG(5, q)$ obtained by adding to $\cS_1$ the set of $q^3+q^2+q$ Greek planes of $\cK$ that are distinct from $\gamma$. Then $\cS_3$ contains $q^6-q^3$ elements disjoint from $\gamma$ and $q^3+q^2+q$ members meeting $\gamma$ in a point. Again, two planes of $\cS_3$ share at most a point.  
\end{cons}

\begin{cons}\label{c3}
Through a line $\ell$ of $\gamma$ there are $q-1$ planes of $\PG(5, q)$ meeting $\cK$ exactly in $\ell$. Selecting one of these planes for each line of $\gamma$, we get a family $\cT$ of $q^2+q+1$ planes mutually intersecting in a point. Let $\cS$ be the set obtained by joining together $\cS_3$ and $\cT$.   
\end{cons}

{\bf In the remaining part of this section we assume that $\cB$ is a circumscribed bundle.} We will show that $\cS_1$ contains $q^3-q^2-q$ subsets each consisting of $q^3-1$ pairwise disjoint planes. In order to obtain this result we will investigate the action of the stabilizer $C$ in $G$ of a point $T \in \PG(3, q) \setminus \pi$ on the hyperbolic quadrics of $\cH$ not containing $T$. The group $C$ has order $q^3-1$, it is the direct product $C_1 \times C_3$, where $C_3 = C \cap C_2$ is the homology group of order $q-1$ having axes $\pi$ and centre $T$ and $C|_{\pi} = C_1$. 

\begin{lemma}\label{lines}
The group $C$ has $q+1$ orbits of size $q^3-1$ on lines not in $\pi$ and not through $T$ and acts semiregularly on the $q^3-1$ points of $\PG(3,q) \setminus (\pi \cup \{T\})$.
\end{lemma}
\begin{proof}
Let $\ell$ be a line not in $\pi$ and not through $T$. It is enough to show that the stabilizer of $\ell$ in $C$ is trivial. Indeed, let $P = \ell \cap \pi$ and let $g \in C$ such that $\ell^g = \ell$, then $P^g = P$ and hence $g \in C_3$. Since $T \notin \ell$, we have that necessarily $g$ is the identity. The second part of the statement easily follows.   
\end{proof}

The following result has been proved in \cite[Section 5]{CP1}.

\begin{lemma}\label{steiner}
Let $A_1, A_2$ two distinct points of $\pi$ and let $g \in C_1$ such that $A_1^g = A_2$. Then $\{\ell \cap \ell^g \;\: | \;\; A_1 \in \ell\}$ is a conic of $\cB$ (passing through $A_1$ and $A_2$). 
\end{lemma}
\begin{proof}
Let $g$ be the unique element of $C_1$ such that $A_1^g = A_2$. Let ${\cal P}_{A_1}$ and ${\cal P}_{A_2}$ be the pencils of lines with vertices $A_1$ and $A_2$. Clearly, $g$ is a projectivity sending ${\cal P}_{A_1}$ to ${\cal P}_{A_2}$ that does not map the line $\ell$ onto itself. In \cite{steiner} it is proved that the set of intersection points of corresponding lines under $g$ is a conic $\cC$ passing through $A_1$ and $A_2$ (Steiner's argument). In particular, the projectivity $g$ maps $t_{A_1}$, the tangent line to $\cC$ at $A_1$, onto the line $A_1 A_2$ and the line $A_1 A_2$ onto $t_{A_2}$, the tangent line to $\cC$ at $A_2$. Embed $\pi\cong \PG(2,q)$ in $\Pi=\PG(2,q^3)$. We denote by $\Delta$ the unique triangle of three points of $\Pi\setminus \pi$ fixed by $C_1$. By considering ${\cal P}_{A_1}$ and ${\cal P}_{A_2}$ as pencils of $\Pi$ and by repeating the previous argument, a conic $\bar \cC$ of $\Pi$ passing through the vertices of $\Delta$ and containing $\cC$ arises. It follows that $\cC$ is a member of the circumscribed bundle $\cal B$ of $\pi$ left invariant by $C_1$. 
\end{proof}

\begin{lemma}\label{tangent}\cite[Lemma 3.2]{CP}
Let $\cC$ and $\cC'$ be two conics of $\cB$ meeting at the point $B$. Then, the tangent lines $t$ and $t'$ to $\cC$ and $\cC'$ at $B$, respectively, must be distinct.
\end{lemma}

\begin{lemma}\label{flags}
The $(q+1)(q^2+q+1)$ flags of $\pi$ are partitioned into $q+1$ orbits of equal size under the action of $C_1$. If $\cC \in \cB$, then each of these orbits contains one element $(P, t_{P})$, where $P \in \cC$ and $t_{P}$ is tangent to $\cC$ at $P$ and the $q$ elements $(A, AP)$, with $A \in \cC \setminus \{P\}$.
\end{lemma}
\begin{proof}
The proof of the first part of the statement is clear. Let $\cC$ be a conic of $\cB$ and let $A_1, A_2$ be two points of $\cC$. From the proof of Lemma \ref{steiner}, the unique element $g$ of $C_1$ mapping $A_1$ to $A_2$ maps the line $A_1 A_2$ to $t_{A_2}$, the tangent line to $\cC$ at $A_2$. Hence the flag $(A_1, A_1 A_2)$ is mapped by $g$ to the flag $(A_2, t_{A_2})$. 
Let $A_3 \in \cC \setminus \{A_1, A_2\}$ and let $g'$ be the unique element of $C_1$ sending $A_1$ to $A_3$. Steiner's argument with $g$ replaced by $g'$, applied to the pencils ${\cal P}_{A_1}$ and ${\cal P}_{A_3}$, gives rise to a conic $\cC'$ that necessarily belongs to ${\cal B}$. Furthermore, being the conic of $\cal B$ through two distinct points of $\pi$ unique, it follows that $\cC=\cC'$. Since $A_1^{g'} = A_3$ and $A_1 \in t_{A_1}$ then $A_1^{g'}= A_3 \in t_{A_1}^{g'}$ and $t_{A_1}^{g'}=A_1A_3$. 
Analogously, the point $Q=(A_1 A_2)^{g'} \cap (A_1 A_2)$ lies on $\cC$ and, of course, also on $A_1 A_2$. Therefore $Q\in\{A_1,A_2\}$. If $Q=A_1$, then $t_{A_1}^{g'} =A_1A_3=(A_2A_1)^{g'}$ and hence $t_{A_1}=A_1A_2$, a contradiction. It follows that $Q=A_2$ and hence $(A_1A_2)^{g'}=A_3A_2$.
\end{proof}

\begin{lemma}\label{bun}
Let $\cC$, $\cC'$ be two distinct conics of $\cB$ and let $g \in C_1$ such that $\cC' = \cC^g$. If $\cC \cap \cC' = \{B\}$ and $A \in \cC$ with $A^g = B$, then $B \in P P^g$, for every point $P \in \cC \setminus \{A, B\}$. In particular, the lines $BB^g$  and $BB^{g^-1}$ are tangent to $\cC$ and $\cC'$ at $B$, respectively.
\end{lemma}
\begin{proof}
Let $P \in \cC \setminus \{A, B\}$. By construction, taking into account Lemma \ref{steiner}, we get $P=(AP)\cap (AP)^g$ and hence $P\in (AP)^g=BP^g$. The second part of the statement easily follows from the proof of Lemma \ref{flags}.
\end{proof}

\begin{prop}\label{regulus}
Let $\cQ$ be a hyperbolic quadric of $\cH$ not containing $T$ and let $\cR$ be one of its reguli. Then no two lines of $\cR$ are in the same $C$--orbit. 
\end{prop}
\begin{proof}
Let $\cR = \{\ell_i \;\; | \;\; 1 \le i \le q+1\}$ and let $\cC = \cQ \cap \pi$ be the conic of the bundle $\cB$. By projecting $\ell_i$ from $T$ onto $\pi$, we get a set $\cF$ consisting of the $q+1$ flags $(B_i, r_i)$, $1 \le i \le q+1$, of $\pi$, where $B_i = \ell_i \cap \pi \in \cC$ and $r_i = \pi \cap \langle T, \ell_i \rangle$. Note that through a point of $\cC$ there pass at most two of the lines $r_1, \dots, r_{q+1}$ and if $r_i \cap \cC = \{B_i\}$, then $r_i$ is the unique among the lines $r_1, \dots, r_{q+1}$ that contains $B_i$. Indeed, if $r_i \cap \cC = \{B_i\}$, then the plane $\langle T, \ell_i \rangle$ contains $r_i$. Hence the plane $\langle T, \ell_i \rangle$ is tangent to $\cQ$ at $B_i$ and the line $B_i T$ is tangent to $\cQ$ at $B_i$. Assume, by contradiction, that there exists an element $\bar{g} \in C$ such that $\ell_i^{\bar{g}} = \ell_j$, with $i \ne j$. Then $B_i^g = B_j$ and $r_i^g = r_j$, where $g = \bar{g}|_{\pi} \in C_1$. Hence the two flags $(B_i,r_i)$ and $(B_j,r_j)$ are in the same $C_1$--orbit. From Lemma \ref{flags} it is not possible that both $r_i$ and $r_j$ are tangent to $\cC$. If $r_i$ is tangent to $\cC$ at $B_i$ then, from Lemma \ref{flags}, we have that $r_j = B_i B_j$, contradicting the fact that $r_i$ is the unique among the lines $r_1, \dots, r_{q+1}$ that contains $B_i$. Analogously if $r_j$ is tangent to $\cC$ at $B_j$. If both $r_i$ and $r_j$ are secant to $\cC$ then, from Lemma \ref{flags}, it follows that there exists a point $B_k \in \cC \setminus \{B_i, B_j\}$ such that $r_i = B_i B_k$ and $r_j = B_j B_k$. In this case we find three distinct lines, $r_i, r_j, r_k$ through $B_k$, contradicting the fact that through a point of $\cC$ there pass at most two of the lines $r_1, \dots, r_{q+1}$.       
\end{proof}

\begin{cor}\label{orbits}
Let $\cQ$ be a hyperbolic quadric of $\cH$ not containing $T$ and let $\cR$ be one of its reguli. Then every line of $\PG(3,q)$ not in $\pi$ and not through $T$ is contained in exactly one regulus of $\cR^C$.
\end{cor}
\begin{proof}
Note that $|\cR^C| = q^3-1$. Indeed, from Lemma \ref{regulus} and Lemma \ref{lines}, we have that $|\cR^C| \ge |\ell^C| = q^3-1$. On the other hand, $|\cR^C| \le |C| = q^3-1$. Assume, by contradiction, that a line $\ell$ is contained in two distinct reguli of $\cR^C$. Since both $\ell^C$ and $\cR^C$ are $C$--orbits and $|\cR^C| = |\ell^C|$, it follows that the number of lines of $\ell^C$ contained in $\cR$ equals the number reguli of $\cR^C$ through $\ell$, contradicting Lemma \ref{regulus}.  
\end{proof}

\begin{prop}\label{symp}
Let $\cQ$ be a hyperbolic quadric of $\cH$ not containing $T$ and let $\cR$ be one of its reguli. Then no two reguli of $\cR^C$ are contained in a $\cW(3,q)$.  
\end{prop}
\begin{proof}
Let $g$ be a non--trivial element of $C$ and consider the two reguli $\cR$ and $\cR^g$ of $\cQ$ and $\cQ^g$, respectively. By Corollary \ref{orbits}, we have that no line is contained in both $\cR$ and $\cR^g$. Assume, by contradiction, that there exists a symplectic polar space $\cW(3,q)$ with symplectic polarity $\perp_s$ such that the $2(q+1)$ lines of $\cR \cup \cR^g$ are lines of $\cW(3,q)$. We consider several cases. 

\medskip
\fbox{$\cQ \cap \pi = \cQ^g \cap \pi = \cC$.} 
\medskip

\noindent
In this case $g \in C_3$. This means that if $\ell$ is a line of $\cR$, then the plane $\langle \ell, \ell^g \rangle$ contains the point $T$. Moreover, if $P \in \cC$ and $\ell \cap \pi = \{P\}$, then $P^{\perp_{s}} = \langle \ell, \ell^g \rangle$ and $T \in P^{\perp_s}$. Therefore, since the points of $\cC$ generate the plane $\pi$, we would have $\pi^{\perp_s} = T$ and hence $T \in \pi$, a contradiction. 

\medskip
\fbox{$\cQ \cap \pi = \cC, \cQ^g \cap \pi = \cC'$, $\cC \ne \cC'$.} 
\medskip

\noindent
Let $\cR^o$ be the opposite regulus of $\cR$ and hence $({\cR}^o)^g$ is the opposite regulus of $\cR^g$, let $\cC \cap \cC' = \{B\}$ and let $\ell, m, \ell', m'$ be the lines of $\cR, \cR^o, \cR^g, (\cR^o)^g$, respectively, through the point $B$. Then $\langle \ell, m \rangle$ is the tangent plane to $\cQ$ at $B$ and $\langle \ell', m' \rangle$ is the tangent plane to $\cQ^g$ at $B$. Note that $\ell \ne \ell'$ and $m \ne m'$, otherwise two reguli of $\cR^C$ or of $(\cR^o)^C$ would share a line contradicting Corollary~\ref{orbits}. Let $\sigma$ be the plane containing $m$ and $m'$ and let $\bar{\sigma}$ be the plane containing $\ell$ and $\ell'$. Thus $B^{\perp_s} = \langle \ell, \ell' \rangle = \bar{\sigma}$. First of all observe that $\sigma \ne \bar{\sigma}$, otherwise the conics $\cC$ and $\cC'$ would admit the same tangent line, namely $\sigma \cap \pi$, at their common point $B$, contradicting Lemma \ref{tangent}.  We consider several cases.

\bigskip

Assume that $\ell' = m$ (which implies $m'\not\subset \bar\sigma$). In this case $\sigma \cap \bar{\sigma} = \ell'$, $\bar{\sigma}$ is tangent to $\cQ$ at $B$ and $\sigma$ is tangent to $\cQ^g$ at $B$. Moreover $\bar{\sigma}$ contains a line of $(\cR^o)^g$, say $t'$, and $\sigma$ contains a line of $\cR$, say $r$. Obviously, $t'\ne \ell'$ and $r\ne \ell$. From Lemma~\ref{bun}, it follows that $t' \cap \cC^g = B^g$, $r \cap \cC = B^{g^{-1}}$ and hence $r^g = \ell'$, $(\ell')^g = t'$. Let $Z = r \cap m'$. Then $Z \ne B$ and since $Z \in \cQ^g$, there exists $v \in \cR^g$ such that $v \cap m' = Z$. Note that both $r, v$ are lines of $\cW(3, q)$ and hence $Z^{\perp_s} = \langle r, v \rangle$. Also $m'^{\perp_s} \subset Z^{\perp_s} = \langle r, v \rangle$, since $Z \in m'$. On the other hand, $B \in m'$ and hence $m'^{\perp_s} \subset B^{\perp_s} = \bar{\sigma}$. Note that $B\in m'$ and $m'\not\subset B^{\perp_s}$, hence $m'$ is not a line of $\cW(3,q)$. Since each of the lines in $\cR^g$ is a line of $\cW(3, q)$ and meets both $m'$ and $t'$, we infer that $m'^{\perp_s} = t'$ and $t' \subset \langle r, v \rangle$. Therefore the point $M=t' \cap r \in\sigma \cap \bar{\sigma} = \ell'$ or, in other words, $M=t' \cap \ell' = r \cap \ell'$. It follows that $(r \cap \ell')^g = r^g \cap \ell'^g = \ell' \cap t' = r \cap \ell'$. Taking into account that $M\notin (\pi \cup \{T\})$, this contradicts Lemma \ref{lines}. A similar argument shows that $\ell$ and $m'$ are distinct lines.

\bigskip

Assume that $\ell \subset \sigma$. Then $\sigma \cap \bar{\sigma} = \ell$ and $B^{\perp_s} = \bar{\sigma}$. Let $n, n'$ be the lines of $\cR^o, (\cR^o)^g$, respectively, that are contained in $\bar{\sigma}$ and let $U = n \cap n'$. Then $U \in \cQ \cap \cQ^g$ and $U \in B^{\perp_s}$.

If $U\notin \ell \cup \ell'$, let $P = n \cap \pi \in \cC$ and $P' = n' \cap \pi \in \cC'$. Then the line $P P'$ is contained in $\sigma$ and $B \in P P'$. Since $\cC' = \cC^g$ and $B \in PP'$, from Lemma~\ref{bun}, we have that $P' = P^g$ and hence $n' = n^g$. Let $t$, $t'$ be the lines of $\cR, \cR^g$, respectively, passing through the point $U$ and let $\tilde{\sigma} = \langle t, t' \rangle$. Thus $U^{\perp_s} = \tilde{\sigma}$ and since $U \in B^{\perp_s}$, we would have that $B \in U^{\perp_s}$, with $B\ne U$. But this means that both $m$ and $m'$ are contained in $\tilde{\sigma}$ and hence $\tilde{\sigma} = \sigma$, contradicting the fact that $U \in \tilde{\sigma} \setminus \sigma$. 

If $U\in\ell$, then, as before, $n'=n^g$ and in the plane $\sigma$ there exists a line $t'$ of $\cR^g$ through $U$, intersecting the conic $\cC'$ at a point, say $P'$. Then $\sigma$ meets the plane $\pi$ in a line which is tangent to $\cC$ in $B$ and secant to $\cC'$ in the points $B$ and $P'$. From Lemma \ref{bun}, it follows that $P'=B^g$ and hence $t' = \ell^g$. Therefore $U^g = \ell^g \cap n^g = t' \cap n' = U$, contradicting Lemma \ref{lines}.

If $U \in \ell'$, note that $B^{\perp_s} = \bar{\sigma}$ and $n, n' \subset \bar{\sigma}$ imply that $B \in n^{\perp_s}$, $B \in n'^{\perp_s}$. On the other hand each of the lines in $\cR$ is a line of $\cW(3, q)$ and meets both $m$ and $n$. Taking into account that $B\in m$ and $m\not\subset B^{\perp_s}$, we get that $m$ is not a line of $\cW(3,q)$. Hence $m^{\perp_s}$  is a line of $\cR^o$ as well and it is disjoint from $m$. Since $n$ is the unique line of $\cR^o$ contained in $\bar\sigma$, we have $m^{\perp_s} = n$. Similarly $m'^{\perp_s} =  n'$. It follows that $\sigma^{\perp_s} = \langle m , m' \rangle^{\perp_s} = m^{\perp_s} \cap m'^{\perp_s} = n \cap n' = U$, that is $U^{\perp_s} = \sigma$. In particular $U \in \sigma$. Since $U \in \ell'$, we would have $U = B$, a contradiction.     

\bigskip

Assume that $\sigma$ does not contain neither $\ell$ nor $\ell'$. Then $\sigma$ contains a line of $\cQ$ and hence is tangent to $\cQ$ at some point distinct from $B$. Analogously $\sigma$ is tangent to $\cQ^g$ at some point distinct from $B$. Let $r, r'$ be the lines of $\cR, \cR^g$, respectively, contained in $\sigma$ and let $U = r \cap r' \in \cQ \cap \cQ^g$. Then  $U^{\perp_s} = \langle r, r' \rangle = \sigma$, with $B \in U^{\perp_s}$ and $B\ne U$. Therefore, $U \in \bar{\sigma} \setminus (\ell \cup \ell')$, with $U \in \cQ \cap \cQ^g$.

Suppose that $U \notin m \cup m'$, i.e. $\sigma\cap\bar\sigma\ne\{m,m'\}$. A similar argument to that used above gives that there are $n, n'$, lines of $\cR^o, (\cR^o)^g$, respectively, that are contained in $\bar{\sigma}$. In particular, $U = n \cap n'$, $m\ne n$ and $m'\ne n'$. Let $P = r \cap \pi \in \cC$ and $P' = r' \cap \pi \in \cC'$. Then the line $P P'$ is contained in $\sigma$ and hence $B \in P P'$. Since $\cC' = \cC^g$ and $B \in PP'$, from Lemma~\ref{bun}, we have that $P' = P^g$ and hence $r' = r^g$. Similarly, let $Q = n \cap \pi \in \cC$ and $Q' = n' \cap \pi \in \cC'$. Then the line $Q Q'$ is contained in $\bar{\sigma}$, hence $B \in Q Q'$, $Q' = Q^g$ and $n' = n^g$. It follows that $U^g \in r^g$ and $U^g \in n^g$, where $r^g \cap n^g = r' \cap n' = U$,  i.e. $U^g=U$, a contradiction. 

If $U \in m$, arguing as above we have $r'=r^g$ and $\sigma \cap \bar\sigma = m$. Also, the plane $\bar\sigma=\langle \ell,m\rangle$ is tangent to $\cQ$ in $B$ and is tangent to $\cQ^g$ since it contains the line $\ell'$. Then there exists a line $t' \in (\cR^o)^g$ such that $t' \subset \bar{\sigma}$, $U \in t'$ and $t'$ intersects the conic $\cC'$ at a point, say $P'$. It follows that $\bar{\sigma}$ intersects the plane $\pi$ at a line which is tangent to $\cC$ at $B$ and secant to $\cC'$ in the points $B$ and $P'$. From Lemma \ref{bun}, we have that $P' = B^g$ and hence $t' = m^g$. Therefore $U^g = r^g \cap m^g = r' \cap t' = U$, a contradiction.

If $U \in m'$, arguing as above we have $r' = r^g$ and $\sigma \cap \bar\sigma = m'$. Also, the plane $\bar\sigma=\langle \ell',m'\rangle$ is tangent to $\cQ^g$ in $B$ and is tangent to $\cQ$ since it contains the line $\ell$. Then there exists a line $t \in \cR^o$ such that $t \subset \bar\sigma$, $U \in t$ and $t$ intersects the conic $\cC$ at a point, say $P$. It follows that $\bar\sigma$ intersects the plane $\pi$ at a line which is tangent to $\cC'$ at $B$ and secant to $\cC$ in the points $B$ and $P$. From Lemma \ref{bun} we have that $P^g = B$ and hence $t^g = m'$. Since $r\ne t$, we get $U=r\cap t$ and hence $U^g=(r\cap t)^g=r^g\cap t^g=r'\cap t'=U$, a contradiction. 
\end{proof}

\begin{theorem}\label{disj}
$\cS_1$ contains $q^3-q^2-q$ subsets each consisting of $q^3-1$ pairwise disjoint planes.
\end{theorem}
\begin{proof}
Let $\cQ$ be a hyperbolic quadric of $\cH$ such that $T \notin \cQ$ and let $\cR$ be one of the reguli of $\cQ$. Then $\cR^C$ is a set of $q^3-1$ reguli of $\PG(3,q)$ and the image of $\cR^C$ under the map $\rho$ is a set of $q^3-1$ conics of $\cK$ lying in $q^3-1$ planes of $\cS_1$. Let $\Sigma$ be the subset of $\cS_1$ consisting of these $q^3-1$ planes. We claim that the $q^3-1$ planes of $\Sigma$ are pairwise disjoint. Let $\cR_1, \cR_2$ be two reguli of $\cR^C$ such that their images under the map $\rho$ are two conics of $\cK$ lying in the two planes $\pi_1, \pi_2$ of $\Sigma$. Then $|\pi_1 \cap \pi_2| \le 1$. If $\pi_1 \cap \pi_2$ is a point, say $P$, then either $P \in \cK$ or $P \notin \cK$. If the former case occurs, then $\rho^{-1}(P)$ is a line of $\PG(3,q)$ contained in both reguli $\cR_1, \cR_2$, contradicting Corollary~\ref{orbits}. If the latter case occurs, then the hyperplane $P^\perp$ coincides with $\langle \pi_1^\perp, \pi_2^\perp \rangle$ and meets $\cK$ in a parabolic quadric $\cQ(4,q)$. In particular $\rho^{-1}(\pi_1^\perp \cap \cK)$ and $\rho^{-1}(\pi_2^\perp \cap \cK)$ are $\cR_1^o$ and $\cR_2^o$, the opposite reguli of $\cR_1$ and $\cR_2$, respectively. Moreover, $\cR_1^o$ and $\cR_2^o$ would be contained in the symplectic polar space $\rho^{-1}(\cQ(4,q))$, contradicting Proposition~\ref{symp}.  

To conclude the proof observe that the group $G$ is transitive on the $q^3$ points of $\PG(3,q) \setminus \pi$ and it is transitive on the $(q^6-q^3)/2$ hyperbolic quadrics of $\cH$. Since a quadric of $\cH$ contains $q^2+q$ points of $\PG(3,q) \setminus \pi$, it follows that there are $(q^2+q)(q^3-1)/2$ quadrics of $\cH$ through the point $T$. Therefore there are $(q^3-q^2-q)(q^3-1)/2$ quadrics of $\cH$ not containing $T$. 
\end{proof}

\subsection{Improving the lower bound on $\cA_q(9, 4; 3)$}

In \cite{BOW}, the authors, by prescribing a subgroup of the normalizer of the Singer group of $\PGL(9, 2)$ as an automorphism group, were able to exhibit, with the aid of a computer, a set of $5986$ planes of $\PG(8, 2)$ mutually intersecting in at most a point. Apart from this example, up to date, if $q \ge 3$, the known lower bound on the maximum number of planes of $\PG(8, q)$ mutually intersecting in at most one point, equals $q^6 \cA_q(6, 4; 3) + 1$, see \cite[Corollary 39]{ST} and \cite[Theorem 2.3]{GT}. This leads to the following lower bound: $\cA_q(9, 4; 3) \ge q^{12}+2q^8+2q^7+q^6+1$. Note that $A_q(9, 4; 3) \le (q^6+q^3+1)(q^2+1)(q^4+1)$ and equality occurs if there exists a collection $\cX$ of planes of $\PG(8, q)$ such that every line of $\PG(8, q)$ is contained in exactly one plane of $\cX$, see \cite{tables}. Here we provide a construction which improves the lower bound on the maximum number of planes of $\PG(8, q)$ pairwise intersecting in at most one point.  

\begin{theorem}
There exists a set of $q^{12} + 2q^8 + 2q^7 + q^6 + q^5 + q^4 +1$ planes of $\PG(8, q)$ mutually intersecting in at most one point.
\end{theorem}
\begin{proof}
In $\PG(8, q)$, let $\Gamma$ be a five--space and let $\gamma$ be a plane disjoint from $\Gamma$. Let $A$ be a collection of $q^6+2q^2+2q+1$ planes of $\Gamma$ obtained as described in Construction \eqref{c3} and let $\alpha_i$, $1 \le i \le q^6+2q^2+2q+1$, be the members of $A$. Let $\Gamma_i$ be the five--spaces containing $\gamma$ and $\alpha_i$. If $i \ne j$, then either $|\alpha_i \cap \alpha_j| = 0$ and $\Gamma_i \cap \Gamma_j = \gamma$, or $\alpha_i \cap \alpha_j$ is a point, say $P$, and $\Gamma_i \cap \Gamma_j$ is the solid $\langle P, \gamma \rangle$. From Theorem~\ref{disj}, there exists a subset $A'$ of $A$ consisting of $q^3-1$ pairwise disjoint planes. Without loss of generality we may assume that $A' = \{\alpha_i \;\; | \;\; 1 \le i \le q^3-1\}$. Let $\cD_1$ be a collection of $q^6+2q^2+2q+1$ planes of $\Gamma_1$ obtained as described in Construction \eqref{c3}, while let $\cD_i$ be a collection of $q^6+q^2+q$ planes of $\Gamma_i$ obtained as described in Construction \eqref{c2}, if $2 \le i \le q^3-1$, and let $\cD_i$ be a collection of $q^6$ planes of $\Gamma_i$ obtained as described in Construction \eqref{c1}, if $q^3 \le i \le q^6+2q^2+2q+1$. Let $\cD = \bigcup_i \cD_i$. We claim that two distinct planes in $\cD$ share at most one point. Let $\pi_j \in \cD_j$ and $\pi_k \in \cD_k$, $\pi_j \ne \pi_k$. If $j = k$, then by construction $|\pi_j \cap \pi_k| \le 1$. Let $j \ne k$ and assume, by contradiction, that $\pi_j \cap \pi_k$ is a line, say $\ell$. Thus $\ell \subset \Gamma_j \cap \Gamma_k$. If $q^3 \le j,k \le q^6+2q^2+2q+1$, then $\Gamma_j \cap \Gamma_k$ meet at most in a solid containing $\gamma$. Hence $|\ell \cap \gamma| \ne 0$, contradicting the fact that both $\pi_j$ and $\pi_k$ are disjoint from $\gamma$. Analogously, if $1 \le j \le q^3-1 \le k \le q^6+2q^2+2q+1$, then $\Gamma_j \cap \Gamma_k$ meet at most in a solid containing $\gamma$ but $\pi_k$ is disjoint from $\gamma$. If $1 \le j, k \le q^3-1$, $k \ne 1$, then $\Gamma_j \cap \Gamma_k = \gamma$. Hence $\ell$ should be contained in $\gamma$, and $\pi_k$ should meet $\gamma$ at least in a line, contradicting the fact that $\pi_k$ meets $\gamma$ at most in one point. Therefore $\cD$ is a set of $q^{12} + 2q^8 + 2q^7 + q^6 + q^5 + q^4 +1$ planes mutually intersecting in at most one point.
\end{proof}

\section{Orbit codes in $\PG(5,q)$}\label{3}

In this section we provide two different construction of a set of $(q^3-1)(q^2+q+1)$ planes of $\PG(5,q)$ that pairwise intersect in at most one point and are permuted in a single orbit by a group of order $(q^3-1)(q^2+q+1)$.

\subsection{Construction using Veronese varieties}

{\bf Assume that $q$ is odd}. Let $X_1, X_2, \dots, X_6$ be homogeneous projective coordinates in $\PG(5,q^3)$, $\cV^*$ the Veronese surface of $\PG(5,q^3)$
\begin{equation}
\label{form:eqV*}\cV^*=\{\langle(x^2, xy, y^{2}, xz, z^2, yz)\rangle:\ x,y,z\in\GF(q^3), (x,y,z)\ne (0,0,0)\}
\end{equation}
and 
\begin{equation}
\label{form:eqM*}\cM^*:\ X_1 X_3 X_5 - X_1 X_6^2 - X_3 X_4^2 - X_5 X_2^2 + 2X_2 X_4 X_6=0
\end{equation} 
its secant variety.

The geometric setting that we will adopt for the proof of the main result is the following representation of $\PG(5,q)$ inside $\PG(5,q^3)$ (in a non--canonical position):
$$
\Sigma=\{\langle (a,b,a^q,b^q,a^{q^2},b^{q^2})\rangle:\ a,b\in\GF(q^3), (a,b)\ne (0,0)\}.
$$ 
It can be seen that $\Sigma$ is fixed by the collineation of order 3 of $\PG(5,q^3)$
$$
\tau:\langle(X_1, X_2, X_3, X_4, X_5, X_6)\rangle \longmapsto \langle(X_5^q, X_6^q, X_1^q, X_2^q, X_3^q, X_4^q)\rangle,
$$ hence $\Sigma\cong \PG(5,q)$.

The Veronese surface $\cV^*$ of $\PG(5,q^3)$ is fixed setwise by $\tau$ and it turns out
\begin{equation}
\label{form:eqV}\cV=\cV^*\cap \Sigma=\{\langle(x^2,x^{q+1},x^{2q},x^{q^2+q},x^{2q^2},x^{q^2+1})\rangle: \ x\in\GF(q^3) \setminus \{0\}\},
\end{equation} 
whereas its secant variety $\cal M$ is represented by the equation
\begin{equation}\label{form:eqM} 
\cM:\ N(a) - Tr(ab^{2q})+2N(b)=0 . 
\end{equation}	
Here $N$ and $Tr$ denote the norm and the trace function from $\GF(q^3)$ to $\GF(q)$, respectively. From \cite[p. 150]{HT}, there are $q^2+q+1$ planes of $\Sigma$ meeting $\cV$ in a non--degenerate conic and these are called {\em conic planes.} From \cite[Theorem 25.2.13]{HT}, $\cM$ is the union of all conic planes. A tangent line lying in a conic plane of $\cV$ is called {\em tangent line} of $\cV$ and all tangent lines of $\cV$ at a point $P \in \cV$ are contained in a plane called {\em tangent plane} of $\cV$ at $P$. Furthermore, the secant variety $\cM$ can also be described as the union of all tangent planes to $\cV$ and a plane of $\Sigma$ meeting $\cV$ exactly in one point and contained in $\cM$ is a tangent plane to $\cV$. The secant (tangent) planes of $\cV$ mutually intersect in one point.

From \cite{BBCE}, the projective space $\Sigma\cong\PG(5,q)$  can be partitioned into two planes, say $\pi_1$ and $\pi_2$, and $q^3-1$ Veronese surfaces. In particular, in this setting the two planes $\pi_1$ and $\pi_2$ can be taken as
$$
\pi_1=\{\langle (a,0,a^q,0,a^{q^2},0)\rangle : \ a\in\GF(q^3) \setminus \{0\}\}, \; \pi_2=\{\langle (0, b, 0, b^q, 0, b^{q^2})\rangle : \ b \in\GF(q^3) \setminus \{0\}\}, 
$$
and the $q^3-1$ Veronese surfaces are 
$$
\cV_{w} = \{\langle(x^2, w x^{q+1},x^{2q}, w^q x^{q^2+q},x^{2q^2},w^{q^2} x^{q^2+1})\rangle: \ x\in\GF(q^3) \setminus \{0\}\},
$$ 
with $w \in \GF(q^3)\setminus \{0\}$. Of course $\cV$ coincides with $\cV_1$. We call $\cP$ such a partition. Note that the group $S$ consisting of the projectivities $g_{\eta}$ of $\Sigma$, with $\eta \in \GF(q^3) \setminus \{0\}$, where $g_\eta$ is associated with the matrix
$$
{\rm diag} \, (\eta^2,\eta^{q+1},\eta^{2q}, \eta^{q^2+q}, \eta^{2q^2}, \eta^{q^2+1}), 
$$ 
has order $q^2+q+1$ and acts regularly on each member of $\cP$. Moreover, the group $T$, consisting of the projectivities $h_w$ of $\Sigma$, with $w \in \GF(q^3) \setminus \{0\}$, where $h_w$ is defined by the matrix
$$
{\rm diag} \, (1, w, 1, w^q, 1, w^{q^2}), 
$$
permutes the $q^3-1$ Veronese surfaces of $\cP$ in a single orbit. Note that $T$ has order $q^3-1$. Consider the plane of $\Sigma$ 
$$
\pi=\left\{\left\langle \left(a, \frac{a+a^q}{2}, a^q, \frac{a^q+a^{q^2}}{2}, a^{q^2}, \frac{a^{q^2} + a}{2}\right)\right\rangle :\ a \in \GF(q^3) \setminus \{0\}\right\},
$$
 and let $\cT = \pi^S$. 
 \begin{lemma}
\begin{itemize}
\item[1)] Every plane of $\cT$ meets $\cV$ in exactly one point.
\item[2)] Every plane of $\cT$ is contained in $\cM$.
\item[3)] Two distinct planes of $\cT$ meet in a unique point.
\item[4)] Three distinct planes of $\cT$ have empty intersection.
\end{itemize}
\end{lemma}
\begin{proof}
{\em 1)} Since $S$ leaves invariant $\cV$, in order to prove the first property it is enough to observe that $\pi \cap \cV = U$, where $U = \langle (1,1,1,1,1,1) \rangle$.

{\em 2)} Taking into account \eqref{form:eqM}, straightforward computations show that the plane $\pi$ is contained in $\cM$. The fact that $S$ fixes $\cM$ yields the result.

{\em 3), 4)} Direct computations show that the planes $\pi$ and $\pi^{g_{\eta}}$, $\eta \notin \GF(q)$, have in common exactly the point $\langle(\eta, (\eta+\eta^q)/2, \eta^q, (\eta^q+\eta^{q^2})/2, \eta^{2q}, (\eta^{q^2}+\eta)/2)\rangle$. Also, the point $\pi \cap \pi^{g_{\eta}}$ lies on the plane $\pi^{g_{\gamma}}$ if and only if $\gamma/\eta \in \GF(q)$, i.e., $\pi^{g_{\eta}} = \pi^{g_{\gamma}}$.
\end{proof} 
 
From \cite[Theorem 25.2.14]{HT}, $\cT$ is the set of tangent planes to $\cV$. Let $G = S \times T$. Then $G$ is a group of order $(q^3-1)(q^2+q+1)$ permuting in a single orbit the points of $\Sigma \setminus (\pi_1 \cup \pi_2)$. Let $\cC = \pi^G$. 

\begin{theorem}
$\cC$ is a $(6, (q^3-1)(q^2+q+1), 4; 3)_q$ constant--dimension subspace code. 
\end{theorem} 
\begin{proof}
Let $g_{\eta} \times h_{w} \in G$. The points of $\pi^{g_{\eta} \times h_{w}} \in \cC$ are 
$$
\left\langle \left(\eta^2 b, \eta^{q+1} w \frac{b+b^q}{2}, \eta^{2q} b^q, \eta^{q^2+q} w^q \frac{b^q+b^{q^2}}{2}, \eta^{2q^2} b^{q^2}, \eta^{q^2+1} w^{q^2} \frac{b^{q^2} + b}{2}\right)\right\rangle ,
$$ 
where $b \in \GF(q^3) \setminus \{0\}$.

First of all observe that $\cC$ consists of $(q^3-1)(q^2+q+1)$ planes. Indeed, if the plane $\pi$ coincides with $\pi^{g_{\eta} \times h_{w}}$, then we would have that for every $a \in \GF(q^3) \setminus \{0\}$, there exists $b \in \GF(q^3) \setminus \{0\}$ such that
$$
\left\{
\begin{array}{l}
a = \eta^2 b\\
\frac{a + a^q}{2} = \eta^{q+1} w \frac{b + b^q}{2}
\end{array} \right. .
$$
Therefore the equality  
\begin{equation} \label{eq1}
(\eta^2 - w \eta^{q+1}) a + (\eta^{2q} - w \eta^{q+1}) a^q = 0 
\end{equation}
should be true for each $a \in \GF(q^3) \setminus \{0\}$. It follows that $\eta^2 \in \GF(q)$. Since $\eta\in\GF(q^3)\setminus\{0\}$, this implies $\eta\in\GF(q)$ and $w = 1$, which means that $g_{\eta} \times h_{w}$ is the identity. 

On the other hand, if $g_{\eta} \times h_{w}$ is not the identity, the planes $\pi$ and $\pi^{g_{\eta} \times h_{w}}$ have at most a point in common. Indeed,
$\pi^{g_{\eta} \times h_{w}}$ meets $\pi$ in the set of points corresponding to the solutions (up to a non--zero scalar in $\GF(q)$) of \eqref{eq1}. Hence, $|\pi \cap \pi^{g_{\eta} \times h_{w}}|$ equals $1$ or $0$, according as $N(w \eta^{q+1} - \eta^2) = N(\eta^{2q} - w \eta^{q+1})$ or not.
\end{proof}

	\begin{cor}
The code $\cC$ is an orbit--code.
	\end{cor}
	\begin{proof}
By construction the code $\cC$ admits the group $G$ of order $(q^3-1)(q^2+q+1)$ as an automorphism group acting transitively on its planes. 
	\end{proof}

\subsection{Construction using Klein quadrics}

Let $\cQ_0$ be the Klein quadric of $\PG(5,q)$ with equation $X_1 X_{4} + X_2 X_{5} + X_3 X_{6} = 0$ and let $\perp_0$ be the associated orthogonal polarity. Consider the two disjoint planes $\pi_1: X_1 = X_2 = X_{3} = 0$ and $\pi_2: X_{4}= X_5 = X_{6}=0$ of $\PG(5,q)$. Then $\pi_i\subset {\cQ_0}$, $i = 1,2$. In particular, we may assume that $\pi_1$ is a Greek plane and $\pi_2$ is a Latin plane. There is a set, say $\cZ$, consisting of $(q+1)(q^2+q+1)$ lines of $\cQ_0$ intersecting both $\pi_1$ and $\pi_2$. Let $P=\langle (x_1, x_2, x_3, 0, 0, 0)\rangle$ be a point of $\pi_2$. Then the hyperplane $P^{\perp_0}$ meets $\pi_1$ in the line 
$$
\ell_P:
\begin{cases} 
x_1X_{4}+x_2X_{5}+x_3X_{6}=0\\
X_1=X_2=X_3=0.
\end{cases}
$$ 
We represent points as row vectors with matrices acting on the right.
Let $C$ be a Singer cycle of ${\rm GL}(3,q)$, and let $H'$ be the cyclic projectivity group generated by the linear collineation of $\PGL(6,q)$ associated with the matrix 
$$
\left(\begin{array}{cc}
C & O_3 \\
O_3 & C^{-T} 
\end{array}\right) ,
$$
where $O_3$ is the zero matrix of order $3$ and for an arbitrary matrix $A$, we denote by $A^{-T}$ its inverse transpose matrix.  Let $H$ be the subgroup of $\PGL(6,q)$
generated by $H'$ and the involution $\tau$ with matrix 
$$
\left(\begin{array}{cc}
I_3 & O_3 \\
O_3 & -I_3 
\end{array}\right) ,
$$
where $I_3$ denotes the identity matrix of order three. It turns out that $H=H'\times \langle\tau\rangle$.
We summarize some results about the group $H$, all of which have been proved in \cite[Proposition 1, Proposition 2, Lemma 8]{BC} with some easy amendments when $q$ is odd. 

\begin{lemma}
The group $H$ satisfies the following properties: 
\begin{itemize}
\item $|H| = q^3-1$;
\item $H$ permutes in a single orbit the points of $\pi_i$, $i =1,2$; 
\item $H$ acts semiregularly on points of $\PG(5,q) \setminus (\pi_1 \cup \pi_2)$; 
\item $H$ stabilizes the hyperbolic quadric $\cQ_0$. 
\end{itemize}
\end{lemma}

Let $\cQ_i$, $0 \le i \le q^2+q$, be the quadrics of $\PG(5,q)$ with associated quadratic form 
$$
{\bf{Q}}_i ({\bf{X}}) = XC^iY^T,
$$ 
where ${\bf X} = (X_1, X_2, X_3, X_4, X_5, X_6)$ and $X=(X_1,X_2,X_3)$, $Y=(X_{4},X_5,X_{6})$. It follows from the definition that if $i \ne j$, then ${\cQ}_i \ne {\cQ}_j$.

We need the following results, which have been proved in \cite[Lemma 3.4, Lemma 3.5, Remark 3.6]{CMP}.
 
\begin{lemma}
\begin{itemize}
\item The lineset $\cZ$ is partitioned into $H$--orbits $\cL_i$ of size $q^2+q+1$, $1 \le i \le q+1$. 
\item For $1 \le i \le q+1$, $\cL_i$ consists of $q^2+q+1$ pairwise disjoint lines. Moreover through a point of $\pi_k$, $k = 1,2$, there passes a unique line of $\cL_i$.
\item For $0 \le j \le q^2+q$, $\cQ_j$ is a non--degenerate hyperbolic quadric of $\PG(5,q)$ left invariant by $H$ and containing $\pi_k$, $k = 1,2$.
\item The set $\cN = \{\cQ_i \;\; 0 \le i \le q^2+q \}$ is a $3$--dimensional linear system of quadrics of $\PG(5,q)$. 
\item Two distinct quadrics of $\cN$ meet in $(q+1)(q^2+q+1)$ points. Moreover these are the points lying on the lines of $\cL_i$ for some $i$. 
\end{itemize}
\end{lemma}

Let $\ell$ be a line of $\PG(5, q)$ disjoint from both $\pi_1$ and $\pi_2$. If $\bar{\cL}_i$ denotes the set of points covered by the lines of $\cL_i$, from \cite[Lemma 3.7]{CMP}, we have the following result.

\begin{lemma}\label{lemma}
If $\ell$ is a line of $\PG(5,q)$ disjoint from both $\pi_1$ and $\pi_2$, then $|\ell \cap \bar{\cL}_i| \le 2$.
\end{lemma}

Assume that $\pi_1$ is a Greek plane and $\pi_2$ is a Latin plane for every quadric of $\cN$. Let $\cG_i$ be the set of $q^3-1$ Greek planes of $\cQ_i$, $0 \le i \le q^2+q$, distinct from $\pi_1$ and disjoint from $\pi_2$. Let $\cC = \bigcup_{i=0}^{q^2+q} \cG_i$.

\begin{theorem}
$\cC$ is a $(6, (q^3-1)(q^2+q+1), 4; 3)_q$ constant--dimension subspace code. 
\end{theorem} 
\begin{proof}
Let $\sigma_1$ and $\sigma_2$ be two distinct planes of $\cC$. If they both belong to a quadric $\cQ_i$ there is nothing to prove. Suppose that $\sigma_1 \subset \cQ_i$ and $\sigma_2 \subset \cQ_j$, $i \ne j$. By way of contradiction let $\ell$ be a line contained in $\sigma_1 \cap \sigma_2$. Then $\ell \subset \bar{\cL}_k$, where $\bar{\cL}_k = \cQ_i \cap \cQ_j$. If $\ell$ were disjoint from both $\pi_1$ and $\pi_2$, then, from Lemma \ref{lemma}, we would have that $|\ell \cap \bar{\cL}_k| \le 2$, a contradiction. If $\ell$ meets $\pi_1$ in one point, say $P_1$, and is disjoint from $\pi_2$, then let $\Sigma$ be the three--space $\langle \pi_1, \ell \rangle$ and let $P_2 = \Sigma \cap \pi_2$. Note that every line meeting non--trivially $\ell$, $\pi_1$, $\pi_2$ passes through $P_2$. But $\cL_k$ contains a unique line through $P_2$ and therefore the line $\ell$ can share at most a further point distinct from $P_1$ with $\bar{\cL}_k$, a contradiction. On the other hand, since every plane of $\cC$ meets $\pi_1$ in exactly one point and is disjoint from $\pi_2$, neither $\ell$ can be contained in $\pi_i$, $i = 1,2$, nor can intersect both planes in a point. This concludes the proof.  
\end{proof}

	\begin{prop}
The group $H$ permutes the $q^3-1$ planes of $\cG_i$ in a single orbit. 	
	\end{prop}
	\begin{proof}
Under the Klein correspondence $\rho$ between lines of $\PG(3,q)$ and points of $\cQ_i$ the points of a Greek plane of $\cQ_i$ correspond to the lines in a plane of $\PG(3,q)$. Since the group $H$ fixes $\cQ_i$, we can look at the action of its corresponding group, say $\bar H$, as a subgroup of $\PGL(4,q)$ acting on $\PG(3,q)$. Hence, the action of $H$ on Greek planes of $\cQ_i$ is equivalent to the action of $\bar H$ on planes of $\PG(3,q)$.  
The group $\bar H$ fixes a point $P$ and a plane $\pi$, with $P\not\in\pi$, induces in $\pi$ a Singer cyclic group and contains the $q-1$ homologies with center $P$ and axis $\pi$. A Greek plane of $\cG_i$ corresponds to a plane of $\PG(3,q)$ not on $P$ and meeting $\pi$ in a line. Observe that $\bar H$ permutes in a single orbit the lines of $\pi$ and the $q-1$ planes through a line of $\pi$ distinct from $\pi$ and not on $P$. Hence the group $\bar H$ permutes the $q^3-1$ planes of $\PG(3,q)$ corresponding to the planes of $\cG_i$ in a single orbit, as required. 	    
	\end{proof}
	
	\begin{cor}
The code $\cC$ is an orbit--code.
	\end{cor}
	\begin{proof}
	
Let $K$ be the cyclic group of $\PGL(6,q)$ generated by the collineation with matrix
$$
\left(\begin{array}{cc}
I_3 & O_3 \\
O_3 & C^{T(q-1)}
\end{array}\right) .
$$
The group $K$ has order $q^2+q+1$. Since ${\bf X}^K=(X,Y)^K=(X,YC^{T(q-1)})$, the group $K$ maps the points of the quadric $\cQ_i$ to the points of the quadric $\cQ_{q^2+2+i}$, with $i\in\{0,\dots,q^2+q\}$,  where indices are taken modulo $q^2+q+1$.  Hence $K$ permutes the quadrics of $\cN$ in a single orbit. The group $K\times H$ has order $(q^3-1)(q^2+q+1)$ and acts transitively on planes of $\cC$.
	\end{proof}	

	\begin{remark}
Note that the two orbit--codes constructed above are not equivalent. It suffices to note that in the first construction the planes involved in the orbit--code are all disjoint from the two planes $\pi_1$ and $\pi_2$ in the mixed partition of $\PG(5,q)$. In the second construction every plane of the orbit--code meets the plane $\pi_1$ in a point and is  disjoint from $\pi_2$.
	\end{remark}	

\bigskip
{\footnotesize
\noindent\textit{Acknowledgments.}
This work was supported 
by the Italian National Group for Algebraic and Geometric Structures and their Applications (GNSAGA-- INdAM).}

\end{document}